\documentclass[10pt]{amsart}
\usepackage{amsmath}
\usepackage{amsfonts}
\usepackage{amssymb}
\usepackage{graphicx}
\usepackage{amsthm,graphicx,color,yfonts}
\usepackage{pdfsync}
\usepackage{epstopdf}%
\usepackage{pifont}
\usepackage{bbm}
\usepackage{a4wide}

\usepackage[colorlinks=true]{hyperref}
\hypersetup{linkcolor=red,citecolor=blue,filecolor=dullmagenta,urlcolor=blue} 

\usepackage{wrapfig}
\usepackage{tikz}
\usetikzlibrary{arrows,calc,decorations.pathreplacing}
\definecolor{light-gray1}{gray}{0.90}
\definecolor{light-gray2}{gray}{0.80}
\definecolor{light-gray3}{gray}{0.60}

\newcommand{\R} {\mathbb R}
\newcommand{\cuad}{{\sqcap\kern-.68em\sqcup}}

\newcommand{\be}{\begin{equation}}
\newcommand{\ee}{\end{equation}}

\definecolor{darkgreen}{rgb}{0.2,0.7,0.1}

\renewcommand{\Re}{\mathrm{Re}}

\newcommand{\T}{\mathbb{T}}

\def\bm{\left( \begin{array}{cc}}

\def\endm{\end{array}\right)}

\newcommand{\ba}{\begin{equation*}}
\newcommand{\ea}{\begin{equation*}}
\newcommand{\bea}{\begin{eqnarray}}
\newcommand{\eea}{\end{eqnarray}}
\newcommand{\bee}{\begin{eqnarray*}}
\newcommand{\eee}{\end{eqnarray*}}
\newcommand{\ben}{\begin{enumerate}}
\newcommand{\een}{\end{enumerate}}

\numberwithin{equation}{section}

\newtheorem{theorem}{Theorem}[section]
\newtheorem*{theorem*}{Theorem}
\newtheorem{proposition}{Proposition}[section]
\newtheorem{corollary}{Corollary}[section]
\newtheorem{lemma}{Lemma}[section]
\newtheorem{definition}{Definition}[section]

\theoremstyle{remark}
\newtheorem{remark}{Remark}[section]

\title[Well-posedness of gBBM]{Well-posedness issues for the generalized Benjamin--Bona--Mahony equation}

\author[S. Kim]{Seunghyun Kim}
\address{Department of Mathematics, Ewha Womans University, Seoul 03760, Korea}
\email{seunghyunkim@ewhain.net}

\author[C. Kwak]{Chulkwang Kwak}
\address{Department of Mathematics, Ewha Womans University, Seoul 03760, Republic of Korea}
\address{Korea Institute for Advanced Study, Seoul 02455, Republic of Korea}
\email{ckkwak@ewha.ac.kr}

\begin{document}

\begin{abstract}
In this paper, we consider the one-dimensional generalized Benjamin--Bona--Mahony (gBBM) equation
\[(1-\partial_x^2)u_t+(u+u^p)_x=0,\qquad p=2,3,4,\dots,\]
posed either on the real line $\mathbb R$ or on the torus $\mathbb T$. This equation may be viewed as a regularized model for the propagation of long-crested surface water waves. The main results of this work are threefold:

\medskip

First, we establish \emph{unconditional local well-posedness} in the class $C([0,T];H^s)$ without imposing any auxiliary spaces for
\[s\ge \frac{p-2}{2p},\]
which is \emph{sharp} in the sense that the multilinear estimate in $H^s$ is optimal. In addition, we prove \emph{unconditional uniqueness} for all distributional solutions in $L^\infty((0,T);H^s)$.

\medskip

Second, we show that below this regularity threshold, the flow map cannot be of class $C^p$. Precisely, if the flow map is well-defined and continuous near the origin from $H^s$ to $C([0,T];H^s)$ for every $s<\frac{p-2}{2p}$, then it cannot be of class $C^p$ at the origin. The proof is based on a high-to-low frequency interaction, implemented differently on $\mathbb R$ and $\mathbb T$.

\medskip

Third, in the odd-power case, we prove \emph{global well-posedness} below $H^1$ in the following cases: $p=3$ with $s\ge \frac14$, and $p=5$ with $s>\frac12$. To the best of our knowledge, these are the first global well-posedness results in the Sobolev framework for the generalized BBM equation below $H^1$. The argument is based on the Bona--Tzvetkov approach \cite{BT}, while being initially inspired by Bourgain's high--low method \cite{Bourgain1998, Bourgain1999}. A key new ingredient is the use of a Hamiltonian conservation law below the $H^1$ energy level. This allows us to control the higher-degree nonlinear contributions in the energy estimate, thereby preventing the Gr\"onwall iteration from blowing up.

\end{abstract}
\maketitle

\section{Introduction}

In this paper, we consider the one-dimensional generalized Benjamin--Bona--Mahony (gBBM) equation \cite{BBM}, also referred to as the regularized long wave equation:
\begin{equation}\label{eq:gBBM}
(1-\partial_x^2)u_t+\left(u+u^p\right)_x=0,\qquad (t,x)\in \mathbb R\times\mathcal M,\qquad p=2,3,4,\dots.
\end{equation}
Here $u=u(t,x)$ is a real-valued scalar function on $\mathcal M=\mathbb R$ or $\mathbb T$. The original BBM equation, corresponding to the case $p=2$, was first derived by Benjamin, Bona, and Mahony \cite{BBM}, and independently by Peregrine \cite{Peregrine}, as a model for the uni-directional propagation of long-crested surface water waves. For more details on the physical relevance of the BBM model, see, for instance, \cite{BPS1981, BCD2005, AABCW2006} and the references therein.

\medskip

Physically, the term $(1-\partial_x^2)u_t$ has a smoothing effect on the temporal evolution of $u$, thereby regularizing the time derivative. This modification suppresses high-frequency instabilities and leads to a more accurate description of low-frequency, long-wave dynamics compared to the classical Korteweg--de Vries (KdV) equation
\[u_t + (u_{xx} + u^2)_x = 0.\]
For this reason, the BBM equation is often preferred in numerical simulations and theoretical studies of water wave propagation. From a mathematical perspective, this regularization is formally obtained via the standard \emph{Boussinesq trick}, which makes the BBM equation more suitable for standard energy methods and facilitates local well-posedness in low-regularity spaces. However, unlike the KdV equation, the BBM equation lacks complete integrability \cite{BPS1980,MMM}.

\medskip

For the original BBM equation ($p=2$), it is well known that the Cauchy problem is globally well-posed in $H^s$ for $s\ge 0$ (see \cite{BT}). Below this threshold, low-regularity obstructions have also been studied in both settings: Panthee \cite{Panthee} proved ill-posedness on $\mathbb R$, while Bona and Dai \cite{BD2017} established norm inflation behavior on $\mathbb T$. For the generalized case with $p=3,4,\dots$, the global well-posedness of \eqref{eq:gBBM} in $H^1$ follows from energy methods based on the $H^1$ conservation law (see \cite{BBM}). The coercivity of the conserved energy functional
\[E[u](t):=\frac12\int_{\mathcal M}\left(u^2+u_x^2\right)(t,x)\,dx\]
provides uniform bounds on the $H^1$-norm of solutions. Together with a standard continuity argument, this allows one to extend local solutions globally in time.

\medskip

While the original BBM equation ($p=2$) has been extensively analyzed in various function spaces \cite{BT, BCH2014, Wang2016, AC2019}, the low-regularity theory for the generalized BBM equation with $p\ge 3$ remains comparatively less developed (see \cite{BC2003}). The present paper fits naturally into this low-regularity theory for the generalized BBM equation. More precisely, we deal with the higher-degree terms on both $\mathbb R$ and $\mathbb T$, and investigate well-posedness issues that include sharp unconditional local well-posedness, a weak ill-posedness below the critical threshold, and global well-posedness in the Sobolev framework below $H^1$ in the odd-power cases $p=3$ and $p=5$. We also remark that, for the original BBM equation ($p=2$), almost sure global well-posedness below the deterministic threshold ($L^2$) has been established by Forlano \cite{Forlano2020}.

\medskip

In the low-regularity analysis, it is convenient to rewrite \eqref{eq:gBBM} in integral form. By Duhamel's principle, \eqref{eq:gBBM} is equivalent to
\begin{equation}\label{eq:Duhamel}
u(t)=S(t)u_0-\int_0^t S(t-t')\phi(\partial_x)\big(u^p(t')\big)\,dt',
\end{equation}
where the linear propagator $S(t)$ is defined on $\mathbb R$ by
\[S(t)f(x)=e^{-t\phi(\partial_x)}f(x)=\frac{1}{\sqrt{2\pi}}\int_{\mathbb R} e^{ix\xi}e^{-it\frac{\xi}{1+\xi^2}}\widehat f(\xi)\,d\xi,\]
and analogously on $\mathbb T$ by Fourier series. Here we define
\begin{equation}\label{eq:diff.oper}
\phi(\partial_x):=(1-\partial_x^2)^{-1}\partial_x.
\end{equation}
On $\mathbb R$ and $\mathbb T$, the operator can be understood as a Fourier multiplier with symbol
\[\phi(\xi)=\frac{i\xi}{1+\xi^2}, \qquad \xi\in\mathbb R,\]
and
\[\phi(n)=\frac{in}{1+n^2}, \qquad n\in\mathbb Z,\]
respectively. We denote by $\widehat f$ the Fourier transform or Fourier coefficient of $f$, depending on the underlying spatial domain. The precise conventions are introduced below.

\medskip

The well-known notion of well-posedness was initially introduced by Hadamard \cite{Hadamard}. The precise statement is as follows:

\begin{definition}[Local well-posedness]\label{def:WP}
Let $u_0 \in H^s(\mathcal M)$ be given. We say that the Cauchy problem for \eqref{eq:gBBM} is locally well-posed in $H^s(\mathcal M)$ if the following conditions are satisfied:
\begin{enumerate}
\item \emph{(Existence)} There exist a time $T=T(\|u_0\|_{H^s(\mathcal M)})>0$ and a subset $X_T^s(\mathcal M)$ of $C([0,T];H^s(\mathcal M))$ such that a solution $u$ to \eqref{eq:gBBM} exists in $X_T^s(\mathcal M)$.

\medskip

\item \emph{(Uniqueness)} The solution is unique in $X_T^s(\mathcal M)$.

\medskip

\item \emph{(Continuous dependence on the data)} The map $u_0\mapsto u$ is continuous from a ball $B\subset H^s(\mathcal M)$ to $X_T^s(\mathcal M)$, where $B$ is endowed with the $H^s(\mathcal M)$ topology.
\end{enumerate}
\end{definition}

\begin{remark}
We say that the Cauchy problem is \emph{unconditionally} well-posed when taking $X_T^s(\mathcal M)=C([0,T];H^s(\mathcal M))$ in Definition \ref{def:WP}.
\end{remark}

\begin{remark}
We say that the Cauchy problem is \emph{globally} well-posed if one can take $T>0$ arbitrarily large.
\end{remark}

\begin{remark}
We say that the Cauchy problem is ill-posed if at least one condition in Definition \ref{def:WP} fails.
\end{remark}

\begin{remark}
In this paper, we say that the Cauchy problem is \emph{weakly ill-posed} if the flow map fails to be analytic.
\end{remark}

We now state the main results of this paper.

\begin{theorem}\label{thm:unconditional_well_posedness}
The gBBM equation \eqref{eq:gBBM} is locally unconditionally well-posed in $H^s(\mathcal M)$ for $s\ge \frac{p-2}{2p}$.
\end{theorem}

\begin{remark}
Local well-posedness for the generalized BBM equation was already established by Bona and Chen \cite{BC2003} in the spaces $L^2\cap L^q$ with $q\ge p$. On the other hand, Theorem \ref{thm:unconditional_well_posedness} is formulated in the Sobolev scale. Note that the regularity threshold appearing in Theorem \ref{thm:unconditional_well_posedness},
\[s\ge \frac{p-2}{2p}=\frac12-\frac1p,\]
exactly corresponds to the natural $L^p$-integrability scale via the Sobolev embedding $H^s(\mathcal M)\hookrightarrow L^p(\mathcal M)$ in one dimension. In this sense, these two local results lie on the same scaling line, while Theorem \ref{thm:unconditional_well_posedness} provides a sharp unconditional result in the Sobolev framework.
\end{remark}

A key ingredient in the proof of Theorem \ref{thm:unconditional_well_posedness} is a multilinear estimate in $H^s(\mathcal M)$. A precise statement is as follows:

\begin{proposition}\label{prop:Hs_multilinear}
Let $p \ge 2$ be an integer. Then, for every $s \ge \frac{p-2}{2p}$, one has
\begin{equation}\label{eq:Hs_multilinear}
\|\phi(\partial_x)(u_1u_2\cdots u_p)\|_{H^s(\mathcal M)} \lesssim \prod_{j=1}^p \|u_j\|_{H^s(\mathcal M)},
\end{equation}
where $\mathcal M=\mathbb R$ or $\mathbb T$, and $\phi(\partial_x)$ is defined as in \eqref{eq:diff.oper}. Moreover, this result is sharp in the sense that \eqref{eq:Hs_multilinear} fails whenever $s<\frac{p-2}{2p}$.
\end{proposition}

The threshold in Theorem \ref{thm:unconditional_well_posedness} is also natural from the viewpoint of distributional solutions, since it is precisely the regularity at which the nonlinear term is well defined in the weak formulation. To make this precise, we first formulate the notion of a distributional solution to the generalized BBM equation \eqref{eq:gBBM}.

\begin{definition}
Let $T>0$ and $s\ge \frac{p-2}{2p}$. A function $u\in L^\infty(0,T;H^s(\mathcal M))$ is called a \emph{distributional solution} to \eqref{eq:gBBM} with initial data $u_0$ if, for every test function $\varphi\in C_c^\infty((-T,T)\times\mathcal M)$, one has
\[\int_0^T\int_{\mathcal M}\Big((\varphi_t+(1-\partial_x^2)^{-1}\partial_x\varphi)\,u+(1-\partial_x^2)^{-1}\partial_x\varphi\cdot u^p\Big)\,dx\,dt+\int_{\mathcal M}\varphi(0,x)u_0(x)\,dx=0.\]
The initial condition is understood in the sense that $u(t)\to u_0$ weakly in $H^s(\mathcal M)$ as $t\to 0^+$.
\end{definition}

\begin{remark}
If $u\in L^\infty(0,T;H^s(\mathcal M))$ with $s\ge \frac{p-2}{2p}$, then, by the Sobolev embedding $H^s(\mathcal M)\hookrightarrow L^p(\mathcal M)$,
\[u^p\in L^\infty(0,T;L^1(\mathcal M))\subset L^1(0,T;L^1(\mathcal M)).\]
Hence the nonlinear term in the weak formulation is well defined. We emphasize that this is precisely the same $L^1$-mechanism used in the proof of the multilinear estimate, where the key step is to control $\|u^p\|_{L^1}$ via the embedding $H^s(\mathcal M)\hookrightarrow L^p(\mathcal M)$.
\end{remark}

\begin{remark}
The notion of unconditional well-posedness, or well-posedness with unconditional uniqueness, goes back to Kato \cite{Kato1995} (see also \cite{GKO2013,Kishimoto2019,MT2022} and the references therein).
\end{remark}

\begin{remark}
If the local well-posedness theorem is proved in $C([0,T];H^s(\mathcal M))$ for some $s$ above this distributional threshold, then unconditional uniqueness means the following: whenever
\[u,v\in L^\infty((0,T);H^s(\mathcal M))\]
are distributional solutions to \eqref{eq:gBBM} with the same initial data, one has $u=v$ on $[0,T]$. In particular, no auxiliary regularity assumptions or additional restrictions on the solution class are imposed.
\end{remark}

We now give a rigorous proof of the unconditional uniqueness statement, using the multilinear estimate established in Proposition \ref{prop:Hs_multilinear}.

\begin{corollary}[Unconditional uniqueness in $L^\infty((0,T);H^s(\mathcal M))$]\label{cor:unconditional_uniqueness}
Let $s\ge \frac{p-2}{2p}$, and suppose that $u,v\in L^\infty((0,T);H^s(\mathcal M))$ are two distributional solutions to the gBBM equation \eqref{eq:gBBM} with the same initial data $u_0$. Then $u=v$ on $[0,T]$.
\end{corollary}

\begin{proof}
Let $w:=u-v$. Since $u$ and $v$ are distributional solutions with the same initial data $u_0$, $w$ satisfies
\[ w_t=-\phi(\partial_x)\left(w+u^p-v^p\right) \]
in the sense of distributions on $(0,T) \times \mathcal M$. Moreover, by the definition of distributional solution,
\[w(t) \rightharpoonup 0 \qquad \text{weakly in } H^s(\mathcal M)\quad\text{as } t\to 0^+.\]
Note that
\[ u^p-v^p=w\sum_{k=0}^{p-1}u^{p-1-k}v^k.\]
Using this, we rewrite the equation above as
\begin{equation}\label{eq:w_eq}
w_t=-\phi(\partial_x)\left(w+w\sum_{k=0}^{p-1}u^{p-1-k}v^k\right).
\end{equation}

Let $J_N$ be a standard Friedrichs mollifier on $\mathcal M$, namely, on $\mathbb R$ we take convolution with a smooth approximate identity, while on $\mathbb T$ we take a standard Fourier truncation. In either case, $J_N$ is bounded on $H^s(\mathcal M)$, commutes with Fourier multipliers, and $J_Nf\to f$ strongly in $H^s(\mathcal M)$ as $N\to\infty$.

Applying $J_N$ to \eqref{eq:w_eq}, we obtain
\[\partial_t(J_Nw) = -\phi(\partial_x)\left(J_Nw+J_N\left(w\sum_{k=0}^{p-1}u^{p-1-k}v^k\right)\right)\]
in the sense of distributions in time with values in $H^s(\mathcal M)$. By Proposition \ref{prop:Hs_multilinear}, the right-hand side belongs to $L^1(0,T;H^s(\mathcal M))$. Hence, by the standard theory of Banach-valued weak derivatives and Bochner--Sobolev spaces (see, for instance, \cite[Chapter III]{Showalter1997} and \cite[Section 6.A]{HunterPDE}), we have
\begin{equation}\label{eq:theory of Banach-valued weak derivatives}
J_Nw\in W^{1,1}(0,T;H^s(\mathcal M))\subset C([0,T];H^s(\mathcal M)).
\end{equation}
Since $w(t)\to 0$ weakly in $H^s(\mathcal M)$ as $t \to 0^+$ and $J_N$ is bounded on $H^s(\mathcal M)$, we have
\[J_Nw(t) \to 0 \quad \text{weakly in } H^s(\mathcal M)\quad \text{as } t \to 0^+.\]
In fact, since $J_Nw\in C([0,T];H^s(\mathcal M))$, we obtain $J_Nw(t)\to J_Nw(0)$ strongly in $H^s(\mathcal M)$ as $t \to 0^+$, which implies $J_Nw(0)=0$. Thus, integrating in time, one obtains for every $t\in[0,T]$ that
\[J_Nw(t) = -\int_0^t \phi(\partial_x)\left(J_Nw(t')+J_N\left(w(t')\sum_{k=0}^{p-1}u(t')^{p-1-k}v(t')^k\right)\right)\,dt'.\]
Taking the $H^s$-norm, using the boundedness of $J_N$ and $\phi(\partial_x)$ on $H^s(\mathcal M)$, and Proposition \ref{prop:Hs_multilinear}, we obtain
\[\begin{aligned}
\|J_Nw(t)\|_{H^s} \lesssim&~{} \int_0^t \|J_Nw(t')\|_{H^s}\,dt' + \sum_{k=0}^{p-1}\int_0^t \|\phi(\partial_x)J_N(wu^{p-1-k}v^k)(t')\|_{H^s}\,dt' \\
\lesssim&~{} \int_0^t \|w(t')\|_{H^s}\,dt' + \sum_{k=0}^{p-1}\int_0^t \|\phi(\partial_x)(wu^{p-1-k}v^k)(t')\|_{H^s}\,dt' \\
\lesssim&~{} \int_0^t \left(1+\sum_{k=0}^{p-1}\|u(t')\|_{H^s}^{\,p-1-k}\|v(t')\|_{H^s}^{\,k}\right)\|w(t')\|_{H^s}\,dt'.
\end{aligned}\]
Recall that $u,v,w \in L^\infty((0,T);H^s(\mathcal M))$. Thus, since $J_Nw(t)\to w(t)$ strongly in $H^s(\mathcal M)$ for a.e. $t\in(0,T)$, it follows, as $N \to \infty$, that
\[\|w(t)\|_{H^s} \lesssim \int_0^t \left(1+\sum_{k=0}^{p-1}\|u(t')\|_{H^s}^{\,p-1-k}\|v(t')\|_{H^s}^{\,k}\right)\|w(t')\|_{H^s}\,dt'\]
for a.e. $t\in(0,T)$. Define
\[F(t):=\|w(t)\|_{H^s}, \qquad G(t):=1+\sum_{k=0}^{p-1}\|u(t)\|_{H^s}^{\,p-1-k}\|v(t)\|_{H^s}^{\,k}.\]
Then $F\in L^\infty(0,T)$, $G\in L^1(0,T)$, and
\[F(t)\lesssim \int_0^t G(t')F(t')\,dt'\]
for a.e. $t\in(0,T)$. By Gr\"onwall's inequality, we conclude that $F(t)=0$ for a.e. $t\in(0,T)$, equivalently,
\[w(t)=0 \qquad \text{for a.e. } t\in(0,T).\]

Finally, \eqref{eq:w_eq} and Proposition \ref{prop:Hs_multilinear} imply that
\[w_t\in L^1(0,T;H^s(\mathcal M)).\]
Therefore, by the same argument as in \eqref{eq:theory of Banach-valued weak derivatives}, we have
\[w \in C([0,T];H^s(\mathcal M)),\]
which, together with $w=0$ for a.e. $t$, implies that $w\equiv 0$ on $[0,T]$, that is, $u=v$ on $[0,T]$.
\end{proof}

Theorem \ref{thm:unconditional_well_posedness} identifies the sharp regularity threshold for the local theory from the perspective of the multilinear estimate and the weak formulation. Our next result shows that below this threshold we indeed encounter an obstruction at the level of the smoothness of the flow map. More precisely, if the flow map is assumed to be well-defined and continuous near the origin, then it cannot be of class $C^p$ below $s=\frac{p-2}{2p}$.

\begin{theorem}\label{thm:weak_ill_posedness}
Let $p \ge 2$ be an integer. For any $s < \frac{p-2}{2p}$, if the flow map $\Phi:u_0\mapsto u$ for the generalized BBM equation \eqref{eq:gBBM} is well-defined and continuous in a neighborhood of the origin as a map from $H^s(\mathcal M)$ to $C([0,T];H^s(\mathcal M))$, then it cannot be of class $C^p$ at the origin.
\end{theorem}

\begin{remark}
In particular, Theorem \ref{thm:weak_ill_posedness} rules out any local theory obtained by the standard Picard iteration method based on the Duhamel formula \eqref{eq:Duhamel}, as long as the associated flow map takes values in $C([0,T];H^s(\mathcal M))$.
\end{remark}

\begin{remark}
Although the proofs on $\mathbb R$ and $\mathbb T$ are implemented differently, the same threshold $s=\frac{p-2}{2p}$ arises in both settings from a common high-to-low frequency interaction. On $\mathbb R$, this mechanism is detected through concentration on a low-frequency interval, whereas on $\mathbb T$ it is captured by an exact discrete mode interaction.
\end{remark}

Although Theorem \ref{thm:unconditional_well_posedness} yields local well-posedness in $H^s(\mathcal M)$ for $s\ge \frac{p-2}{2p}$, extending these solutions globally in time requires additional \emph{a priori} bounds on the solutions. Below $H^1$, the standard conserved quantities do not directly control the $H^s$-norm, so a direct energy argument is not sufficient to obtain global existence. In the odd-power case, however, we obtain global well-posedness below $H^1$ for $p=3$ and $p=5$ by combining a frequency decomposition argument with the use of a Hamiltonian conservation law below the $H^1$ energy level.

\begin{theorem}\label{thm:GWP}
Assume that $p$ is odd. Then the generalized BBM equation \eqref{eq:gBBM} is globally well-posed in $H^s(\mathcal M)$ in the following cases:
\begin{enumerate}
\item if $p=3$, then for $s\ge \frac14$,
\item if $p=5$, then for $s>\frac12$.
\end{enumerate}
\end{theorem}

\begin{remark}
The novelty in Theorem \ref{thm:GWP} lies in the low-regularity Sobolev regime
\[p=3, \quad s\ge \frac14, \qquad p=5, \quad s>\frac12,\]
since the case $s\ge 1$ follows from the classical energy method. We also note that earlier global well-posedness results for the generalized BBM equation by Bona and Chen \cite{BC2003} were obtained in $L^2\cap L^{p+1}$ naturally associated with the Hamiltonian structure. In contrast, Theorem \ref{thm:GWP} is formulated in the Sobolev framework and yields global well-posedness for Sobolev initial data below $H^1$.
\end{remark}

\begin{remark}[Sketch of proof]
The generalized BBM equation \eqref{eq:gBBM} admits the conserved quantity
\begin{equation}\label{eq:Hamiltonian}
M[u](t)=\frac12\int_{\mathcal M} u^2(t,x)\,dx+\frac{1}{p+1}\int_{\mathcal M} u^{p+1}(t,x)\,dx,
\end{equation}
which is preserved for smooth solutions and extends by continuity to the corresponding low-regularity range under consideration. Although this functional does not directly control the $H^s$-norm when $s<1$, it plays a crucial role when combined with a frequency decomposition argument. More precisely, the proof is based on the Bona--Tzvetkov approach \cite{BT}, while being initially inspired by Bourgain's high--low method \cite{Bourgain1998,Bourgain1999}; see also \cite{CKSTT2003} for a refined formulation of this general frequency decomposition philosophy.

\medskip

In this approach, the initial data $u_0 \in H^s(\mathcal M)$ are decomposed into a regular part $w_0 \in H^1(\mathcal M)$ and a small rough part $v_0 \in H^s(\mathcal M)$. The rough component $v$, associated with $v_0$, remains small in $H^s(\mathcal M)$ on a sufficiently long time interval, while the regular component $w$ satisfies a perturbed equation whose $H^1$-norm can be controlled by a Gr\"onwall-type argument.

\medskip

For $p=3$, the threshold $s\ge \frac14$ is closely tied to the conserved quantity $M[u]$, which directly yields an $L^4$ bound on the full solution $u$ and is sufficient to close the perturbative energy estimate. For $p=5$, the conserved quantity still provides the essential $L^6$ control, but this itself is not sufficient. In this case, the argument also requires $L^\infty$ control of the rough component $v$, and this imposes the condition $s>\frac12$ via Sobolev embedding. See Section \ref{sec:GWP} for further details.
\end{remark}

\begin{remark}[Limitations of the present global argument]
The global result (Theorem \ref{thm:GWP}) is currently restricted to only a few odd-power cases, namely $p=3$ and $p=5$, for at least two reasons. The first is that, below $H^1$, the only available conserved quantity is the Hamiltonian \eqref{eq:Hamiltonian}. When $p$ is odd, the exponent $p+1$ is even, so the second term in \eqref{eq:Hamiltonian} is nonnegative, and thus it controls both the $L^2$- and $L^{p+1}$-norms. In contrast, when $p$ is even, the exponent $p+1$ is odd, and the term $\int_{\mathcal M}u^{p+1}\,dx$ is no longer sign-definite. In particular, \eqref{eq:Hamiltonian} cannot control the $L^{p+1}$-norm, so our argument is not available for the even-power cases.

\medskip

Even in the odd-power case, \eqref{eq:Hamiltonian} is limited to controlling the $L^2$- and $L^{p+1}$-norms. In the perturbative $H^1$-energy estimate, one encounters terms requiring higher integrability of $u$, and for $p\ge 7$ this already exceeds the control provided by \eqref{eq:Hamiltonian}. This explains both why the present method closes only for $p=3$ and $p=5$, and why there remains a gap between the sharp local threshold $s\ge \frac{p-2}{2p}$ and the currently available global theory.
\end{remark}

\begin{remark}[Open problems]
From our results, several natural questions remain open. First, below the threshold $s<\frac{p-2}{2p}$, our result gives only weak ill-posedness, that is, the failure of $C^p$ differentiability of the flow map. It would be interesting to determine whether stronger forms of ill-posedness hold in this range, such as lack of uniform continuity, lack of continuity, or norm inflation. Second, it remains open to improve the current global theory toward the sharp local threshold. Finally, extending the low-regularity global theory in the Sobolev framework to higher odd-powers $p\ge 7$ or to even-powers requires new \emph{a priori} estimates or a different global argument.
\end{remark}

\subsection*{Notations}
For nonnegative quantities $x$ and $y$, we write $x \lesssim y$ if there exists a constant $C>0$ such that $x \le Cy$, and $x \sim y$ if $x \lesssim y$ and $y \lesssim x$. We also use the standard notation
\[\langle \zeta\rangle:=(1+|\zeta|^2)^{1/2}.\]

For simplicity, we use $\mathcal F$ to denote both the Fourier transform on $\mathbb R$ and the Fourier series transform on $\mathbb T$, depending on the underlying spatial domain, and we also write $\widehat{f}$ for the corresponding transform. More precisely, for $f\in \mathcal S'(\mathbb R)$,
\[\widehat{f}(\xi)=\frac{1}{\sqrt{2\pi}}\int_{\mathbb R}e^{-ix\xi}f(x)\,dx,\qquad \xi\in\mathbb R,\]
while for $f\in \mathcal D'(\mathbb T)$,
\[\widehat{f}(n)=\frac{1}{\sqrt{2\pi}}\int_{\mathbb T}e^{-inx}f(x)\,dx,\qquad n\in\mathbb Z.\]

For $s\in\mathbb R$, the inhomogeneous Sobolev space $H^s(\mathcal M)$ is defined by the norm
\[\|f\|_{H^s(\mathbb R)}:=\|\langle \xi\rangle^s\widehat{f}(\xi)\|_{L^2_\xi},\qquad \|f\|_{H^s(\mathbb T)}:=\|\langle n\rangle^s\widehat{f}(n)\|_{\ell^2_n}.\]

When needed, we write $P_{\leq L}$ for the Fourier projection onto $\{|\xi|\le L\}$ on $\mathbb R$, and analogously for the corresponding frequency truncation on $\mathbb T$.

\subsection*{Organization of the paper}
In Section \ref{sec:WP}, we prove the unconditional local well-posedness result, Theorem \ref{thm:unconditional_well_posedness}. The key ingredient is the sharp multilinear estimate in Proposition \ref{prop:Hs_multilinear}. In Section \ref{sec:IP}, we prove Theorem \ref{thm:weak_ill_posedness}, showing that the flow map cannot be of class $C^p$ below the critical threshold. Finally, in Section \ref{sec:GWP}, we establish the low-regularity global well-posedness result, Theorem \ref{thm:GWP}.

\subsection*{Acknowledgments} C. K. was partially supported by Young Research Program of the National Research Foundation of Korea(NRF) grant funded by the Korea government(MSIT) (No. RS-2023-00210210) and Global - Learning \& Academic research institution for Master’s·PhD students, and Postdocs(G-LAMP) Program of the National Research Foundation of Korea(NRF) grant funded by the Ministry of Education(No. RS-2025-25442252).

\section{Proof of Theorem \ref{thm:unconditional_well_posedness}}\label{sec:WP}

In this section, we prove Theorem \ref{thm:unconditional_well_posedness}. First, we establish the sharp multilinear estimate stated in Proposition \ref{prop:Hs_multilinear}. Next, we use this estimate to construct local solutions by a contraction mapping argument in $C([0,T];H^s(\mathcal M))$. 

\subsection{Sharp multilinear estimate}

For convenience, we restate Proposition \ref{prop:Hs_multilinear}.

\begin{proposition}\label{prop:multilinear_sharp}
Let $p \ge 2$ be an integer. Then, for every $s \ge \frac{p-2}{2p}$, one has
\begin{equation}\label{eq:multilinear_estimate}
\|\phi(\partial_x)(u_1u_2\cdots u_p)\|_{H^s(\mathcal M)} \lesssim \prod_{j=1}^p \|u_j\|_{H^s(\mathcal M)},
\end{equation}
where $\mathcal M=\mathbb R$ or $\mathbb T$, and $\phi(\partial_x)=(1-\partial_x^2)^{-1}\partial_x$. Moreover, this estimate is sharp in the sense that \eqref{eq:multilinear_estimate} fails whenever $s<\frac{p-2}{2p}$.
\end{proposition}

\begin{proof}
We first prove \eqref{eq:multilinear_estimate}. Recall that the Fourier symbol of $\phi(\partial_x)$ is
\[\phi(\xi)=\frac{i\xi}{1+\xi^2}\]
on $\mathbb R$, and similarly
\[\phi(n)=\frac{in}{1+n^2}\]
on $\mathbb T$. In either case,
\[\langle \cdot\rangle^s |\phi(\cdot)| \lesssim \langle \cdot\rangle^{s-1},\]
and hence
\begin{equation}\label{eq:boundedness of phi}
\|\phi(\partial_x)f\|_{H^s(\mathcal M)} \lesssim \|f\|_{H^{s-1}(\mathcal M)}.
\end{equation}
First we assume $s\ge \frac12$. If $s>\frac12$, then $H^s(\mathcal M)$ is an algebra, and hence by \eqref{eq:boundedness of phi},
\[\|\phi(\partial_x)(u_1u_2\cdots u_p)\|_{H^s(\mathcal M)} \lesssim \|u_1\cdots u_p\|_{H^{s-1}(\mathcal M)} \le \|u_1\cdots u_p\|_{H^s(\mathcal M)} \lesssim \prod_{j=1}^p \|u_j\|_{H^s(\mathcal M)}.\]
At the endpoint $s=\frac12$, although $H^{\frac12}(\mathcal M)$ is not an algebra, we still have
\[\|\phi(\partial_x)f\|_{H^{\frac12}(\mathcal M)} \lesssim \|f\|_{L^2(\mathcal M)},\]
since
\[\sup_{\zeta}\langle \zeta\rangle^{\frac12}|\phi(\zeta)|<\infty,\]
where $\zeta=\xi$ on $\mathbb R$ and $\zeta=n$ on $\mathbb T$. Therefore, by the H\"older inequality and the embedding
\[H^{\frac12}(\mathcal M)\hookrightarrow L^{2p}(\mathcal M)\]
in one dimension, we obtain
\[\|\phi(\partial_x)(u_1u_2\cdots u_p)\|_{H^{\frac12}(\mathcal M)} \lesssim\|u_1\cdots u_p\|_{L^2(\mathcal M)} \le \prod_{j=1}^p \|u_j\|_{L^{2p}(\mathcal M)} \lesssim \prod_{j=1}^p \|u_j\|_{H^{\frac12}(\mathcal M)}.\]
Thus \eqref{eq:multilinear_estimate} holds for $s\ge \frac12$.

Next, assume
\[\frac{p-2}{2p} \le s < \frac12.\]
Since $2s-2<-1$, we have, on both $\mathbb R$ and $\mathbb T$, that
\[\|f\|_{H^{s-1}(\mathcal M)} \lesssim \|f\|_{L^1(\mathcal M)}.\]
Indeed, on $\mathbb R$ this follows from the bound $|\widehat f(\xi)|\le \|f\|_{L^1}$ together with the integrability of $\langle \xi\rangle^{2s-2}$, while on $\mathbb T$ it follows from $|\widehat f(n)|\le \|f\|_{L^1}$ together with the summability of $\langle n\rangle^{2s-2}$. Therefore, by the H\"older inequality and the Sobolev embedding
\[H^s(\mathcal M)\hookrightarrow L^p(\mathcal M),\]
which is valid in one dimension since
\[s\ge \frac12-\frac1p=\frac{p-2}{2p},\]
we obtain
\[\|\phi(\partial_x)(u_1\cdots u_p)\|_{H^s(\mathcal M)} \lesssim \|u_1\cdots u_p\|_{H^{s-1}(\mathcal M)} \lesssim \|u_1\cdots u_p\|_{L^1(\mathcal M)} \le \prod_{j=1}^p \|u_j\|_{L^p(\mathcal M)} \lesssim \prod_{j=1}^p \|u_j\|_{H^s(\mathcal M)}.\]

This proves \eqref{eq:multilinear_estimate} for all $s \ge \frac{p-2}{2p}$.

\medskip

We now prove sharpness. A dyadic decomposition shows that the worst contribution comes from high-to-low frequency interactions, in which several high-frequency inputs combine to produce a low-frequency output. This motivates the following counterexample, where we choose balanced high-frequency inputs whose interaction contains frequencies near the origin.

We first consider the case $\mathcal M=\mathbb R$. Let $N \gg 1$, and define $u_1, \dots, u_p$ by
\[\widehat{u}_j(\xi)=\chi_{[N,2N]}(\xi), \quad 1\le j\le p-1, \quad \text{and} \quad \widehat{u}_p(\xi)=\chi_{[-2(p-1)N,-(p-1)N]}(\xi).\]
Then
\[\|u_j\|_{H^s(\mathbb R)} \sim N^{s+\frac12}, \qquad 1\le j\le p,\]
and thus
\[\prod_{j=1}^p \|u_j\|_{H^s(\mathbb R)} \sim N^{ps+\frac p2}.\]

Set $U=u_1u_2\cdots u_p$. Then
\[\widehat U(\xi)=(\widehat{u}_1*\cdots *\widehat{u}_p)(\xi).\]
Performing the change of variables $\xi_j=N\eta_j$ for $1\le j\le p-1$ in the convolution integrals, one has
\[\widehat U(\xi)=N^{p-1}\Psi\left(\frac{\xi}{N}\right),\]
where
\[\Psi=\underbrace{\chi_{[1,2]}*\cdots *\chi_{[1,2]}}_{p-1 \text{ copies}}*\chi_{[-2(p-1),-(p-1)]}.\]
Suppose $\Psi(0)>0$. Since $\Psi$ is the convolution of compactly supported $L^1$ functions, it is continuous. Hence there exists $\varepsilon>0$ such that
\[\Psi(\eta)\ge c_0>0 \quad \text{for all } |\eta|\le \varepsilon.\]
Hence, for $|\xi| \le 1$ and $N$ sufficiently large,
\[\widehat U(\xi)=N^{p-1}\Psi\left(\frac{\xi}{N}\right)\gtrsim N^{p-1}.\]
Using $|\phi(\xi)| \sim |\xi|$ for $|\xi| \le 1$, we obtain
\[\|\phi(\partial_x)U\|_{H^s(\mathbb R)}^2 \gtrsim \int_{|\xi| \le 1}\langle \xi\rangle^{2s}|\xi|^2|\widehat U(\xi)|^2\,d\xi \gtrsim N^{2p-2}\int_{|\xi|\le1}|\xi|^2\,d\xi \sim N^{2p-2}.\]
Thus,
\[\|\phi(\partial_x)U\|_{H^s(\mathbb R)} \gtrsim N^{p-1}.\]
If \eqref{eq:multilinear_estimate} holds for some $s<\frac{p-2}{2p}$, then we immediately have
\[N^{p-1} \lesssim N^{ps+\frac p2},\]
which contradicts
\[p-1>ps+\frac p2\]
for large $N$. To complete the proof, we now prove $\Psi(0)>0$. Set
\[J_j = \left[1,\frac32\right], \quad 1\le j\le p-1, \quad \text{and} \quad J_p = \left[-\frac32(p-1),-(p-1)\right].\]
Then $J_j \subset [1,2]$ for $1\le j\le p-1$ and $J_p\subset [-2(p-1),-(p-1)]$. Moreover, whenever $\eta_j \in J_j$ for $1\le j\le p-1$, one has
\[-\sum_{j=1}^{p-1}\eta_j \in J_p.\]
Thus
\[\Psi(0) \ge \int_{J_1 \times \cdots \times J_{p-1}} 1\,d\eta_1 \cdots d\eta_{p-1} = \left(\frac12\right)^{p-1} > 0,\]
which proves our claim. Therefore, \eqref{eq:multilinear_estimate} fails on $\mathbb R$ whenever $s<\frac{p-2}{2p}$.

\medskip

We next consider the case $\mathcal M=\mathbb T$. Let $N\gg1$ be an integer, and define $u_1,\dots,u_p$ by
\[\widehat{u}_j(n)=\mathbf 1_{\{N,\dots,2N\}}(n), \quad 1\le j\le p-1, \quad \text{and} \quad \widehat{u}_p(n)=\mathbf 1_{\{1-2(p-1)N,\dots,1-(p-1)N\}}(n).\]
Again,
\[\|u_j\|_{H^s(\mathbb T)} \sim N^{s+\frac12}, \qquad 1\le j\le p,\]
and hence
\[\prod_{j=1}^p \|u_j\|_{H^s(\mathbb T)} \sim N^{ps+\frac p2}.\]
Set $U=u_1u_2\cdots u_p$. Then
\[\widehat U(1)=\sum_{n_1+\cdots+n_p=1}\widehat{u}_1(n_1)\cdots \widehat{u}_p(n_p).\]
Suppose $|\widehat U(1)|\gtrsim N^{p-1}$. Since $|\phi(1)|=\frac12$, we obtain
\[\|\phi(\partial_x)U\|_{H^s(\mathbb T)}^2 \ge \langle 1\rangle^{2s}|\phi(1)|^2|\widehat U(1)|^2 \gtrsim N^{2p-2}.\]
Thus,
\[\|\phi(\partial_x)U\|_{H^s(\mathbb T)} \gtrsim N^{p-1}.\]
If \eqref{eq:multilinear_estimate} holds for some $s<\frac{p-2}{2p}$, then we immediately have
\[N^{p-1} \lesssim N^{ps+\frac p2},\]
which is a contradiction for large $N$. To complete the proof, we now prove $|\widehat U(1)|\gtrsim N^{p-1}$. Set $M=\lfloor N/4\rfloor$. For every choice of integers $n_1,\dots,n_{p-1} \in \{N,\dots,N+M\}$, define
\[n_p = 1-\sum_{j=1}^{p-1}n_j.\]
Then
\[1-(p-1)(N+M) \le n_p \le 1-(p-1)N.\]
Since $M\le N$, it follows that
\[1-(p-1)(N+M) \ge 1-2(p-1)N,\]
and hence
\[n_p \in \{1-2(p-1)N,\dots,1-(p-1)N\}.\]
This says that for each such choice of $(n_1,\dots,n_{p-1})$, all frequencies $n_1,\dots,n_p$ lie in the corresponding supports of $\widehat{u}_1, \dots, \widehat{u}_p$ and satisfy $n_1 + \cdots + n_p = 1$. Thus every such $(n_1,\dots,n_{p-1})$ contributes a nonzero term to $\widehat U(1)$. Since $M+1 \sim N$, we conclude that
\[|\widehat U(1)|\ge (M+1)^{p-1} \sim N^{p-1},\]
which proves our claim. Therefore, \eqref{eq:multilinear_estimate} also fails on $\mathbb T$ whenever $s<\frac{p-2}{2p}$.

This completes the proof.
\end{proof}

\subsection{Local well-posedness}

\begin{proof}[Proof of Theorem \ref{thm:unconditional_well_posedness}]
Let $s\ge \frac{p-2}{2p}$ and let $u_0 \in H^s(\mathcal M)$. We use Duhamel's formula \eqref{eq:Duhamel} and define the flow map
\[\Phi_{u_0}(u)(t) := S(t)u_0-\int_0^t S(t-t')\phi(\partial_x)\bigl(u^p(t')\bigr)\,dt'.\]
We write
\[\|u\|_{C_TH^s}:=\sup_{t\in[0,T]}\|u(t)\|_{H^s(\mathcal M)}.\]
Set
\[R:=\|u_0\|_{H^s(\mathcal M)}, \qquad \mathcal B:=\left\{u\in C([0,T];H^s(\mathcal M)):\|u\|_{C_TH^s}\le 2R\right\}.\]
Since $S(t)$ is unitary on $H^s(\mathcal M)$, we have
\[\|S(t)u_0\|_{C_TH^s} = \|u_0\|_{H^s(\mathcal M)} = R.\]
Moreover, by Proposition \ref{prop:multilinear_sharp},
\[\|\phi(\partial_x)(u^p)\|_{H^s(\mathcal M)} \lesssim \|u\|_{H^s(\mathcal M)}^p\]
for all $u \in H^s(\mathcal M)$. Therefore, for $u \in \mathcal B$,
\[\begin{aligned}
\|\Phi_{u_0}(u)\|_{C_TH^s} \le&~{} \|S(t)u_0\|_{C_TH^s}  + \sup_{t\in[0,T]}\int_0^t \|S(t-t')\phi(\partial_x)(u^p)(t')\|_{H^s}\,dt' \\
\le&~{} R+\int_0^T \|\phi(\partial_x)(u^p)(t')\|_{H^s}\,dt' \\
\lesssim&~{} R+T\|u\|_{C_TH^s}^p \\
\le&~{} R+C\,T(2R)^p
\end{aligned}\]
for some constant $C>0$ depending only on $p$ and $s$. Hence $\Phi_{u_0}$ maps $\mathcal B$ into itself provided that
\[C\,T(2R)^{p-1}\le \frac12.\]

Next, let $u,v\in \mathcal B$. Using
\[u^p-v^p=(u-v)\sum_{k=0}^{p-1}u^{p-1-k}v^k\]
and Proposition \ref{prop:multilinear_sharp}, we obtain
\[\begin{aligned}
\|\Phi_{u_0}(u)-\Phi_{u_0}(v)\|_{C_TH^s} \le&~{} \sup_{t\in[0,T]}\int_0^t \|\phi(\partial_x)(u^p-v^p)(t')\|_{H^s}\,dt' \\
\le&~{} T\,\|\phi(\partial_x)(u^p-v^p)\|_{C_TH^s} \\
\le&~{} C\,T\sum_{k=0}^{p-1}\|u\|_{C_TH^s}^{\,p-1-k}\|v\|_{C_TH^s}^{\,k}\|u-v\|_{C_TH^s} \\
\le&~{} C\,T\,p(2R)^{p-1}\|u-v\|_{C_TH^s}.
\end{aligned}\]
Thus $\Phi_{u_0}$ is a contraction on $\mathcal B$ provided that
\[C\,T\,p(2R)^{p-1}<1.\]

Choosing $T>0$ so that, for instance,
\[C\,T\,p(2R)^{p-1}\le \frac12,\]
both conditions are satisfied. Therefore, by the Banach fixed point theorem, $\Phi_{u_0}$ has a unique fixed point in $\mathcal B$. This yields a unique local solution
\[u\in C([0,T];H^s(\mathcal M))\]
to \eqref{eq:gBBM}. The same contraction argument also gives continuous dependence on the initial data.
\end{proof}

\section{Proof of Theorem \ref{thm:weak_ill_posedness}}\label{sec:IP}

The proof follows the standard argument (see, for instance, Bona--Tzvetkov \cite{BT} and Tzvetkov \cite{Tzvetkov1999}) to show the failure of higher Fr\'echet differentiability of dispersive flow maps. We first record the general consequence of assuming that the flow map is of class $C^p$ near the origin.

\begin{lemma}\label{lem:Cp_expansion}
Let $\mathcal M=\mathbb R$ or $\mathbb T$, and assume that the flow map $\Phi$ from $H^s(\mathcal M)$ to $C([0,T];H^s(\mathcal M))$ is of class $C^p$ near the origin. For $\psi\in H^s(\mathcal M)$ and sufficiently small $\delta\in\mathbb R$, define
\[u_\delta(t):=\Phi(\delta\psi)(t).\]
Then the map $\delta\mapsto u_\delta$ is $C^p$ from a neighborhood of $0$ in $\mathbb R$ into $C([0,T];H^s(\mathcal M))$, and admits the expansion
\[u_\delta(t)=\sum_{k=1}^p \delta^k u^{(k)}(t)+o(|\delta|^p)\quad \text{in } C([0,T];H^s(\mathcal M)),\]
where
\[u^{(k)}(t):=\frac1{k!}D^k\Phi(0)\underbrace{[\psi,\dots,\psi]}_{k \;\mathrm{copies}},\]
and $D^k\Phi(0)$ denotes the $k$-th Fr\'echet derivative of $\Phi$ at the origin. Moreover,
\begin{equation}\label{eq:frechet_bound_common}
\|u^{(p)}\|_{C([0,T];H^s(\mathcal M))}\lesssim \|\psi\|_{H^s(\mathcal M)}^p,
\end{equation}
and the coefficients satisfy
\begin{equation}\label{eq:up_formula_common}
u^{(1)}(t) = S(t)\psi,\quad u^{(k)}(t) \equiv 0 \text{ for } 1 < k < p,\quad u^{(p)}(t) = -\int_0^t S(t-t')\phi(\partial_x)\left((S(t')\psi)^p\right)\,dt'.
\end{equation}
\end{lemma}

\begin{proof}
The Taylor expansion and the bound \eqref{eq:frechet_bound_common} follow immediately from the assumption that $\Phi$ is of class $C^p$ near the origin. It remains to identify the coefficients. By \eqref{eq:Duhamel}, we write
\[u_\delta(t)=\delta S(t)\psi-\int_0^t S(t-t')\phi(\partial_x)\left(u_\delta(t')^p\right)\,dt'.\]
Differentiating with respect to $\delta$ at $\delta=0$, and using that $u_\delta|_{\delta=0}\equiv0$, we obtain
\[u^{(1)}(t)=S(t)\psi,\qquad u^{(k)}(t)\equiv0\ \text{for }1<k<p,\]
and
\[u^{(p)}(t)=-\int_0^t S(t-t')\phi(\partial_x)\left((S(t')\psi)^p\right)\,dt'.\]
This proves \eqref{eq:up_formula_common}.
\end{proof}

\subsection{Lack of $C^p$-differentiability on $\mathbb R$}

We prove by contradiction. Assume that $\Phi$ is a map from $H^s(\mathbb R)$ to $C([0,T];H^s(\mathbb R))$ of class $C^p$ near the origin. Let $N \gg 1$, and let $u_1, \dots, u_p$ be the functions used in the sharpness part of the proof of Proposition \ref{prop:multilinear_sharp}, namely
\[\widehat{u}_j=\chi_{[N,2N]}, \quad 1 \le j \le p-1,\quad \widehat{u}_p=\chi_{[-2(p-1)N,-(p-1)N]}.\]
Define
\begin{equation}\label{eq:v_Nw_N}
v_N := (p-1)u_1, \quad w_N := u_p,\quad \psi_N:= N^{-s-\frac12}(v_N+w_N).
\end{equation}
It follows that
\[\|\psi_N\|_{H^s(\mathbb R)} \sim 1\]
uniformly in $N$.

Let
\[u_N^{(p)}(t):=\frac1{p!}D^p\Phi(0)\underbrace{[\psi_N,\dots,\psi_N]}_{p \;\mathrm{copies}}.\]
By Lemma \ref{lem:Cp_expansion},
\[u_N^{(p)}(t)=-\int_0^t S(t-t')\phi(\partial_x)\left((S(t')\psi_N)^p\right)\,dt'.\]
Using the binomial theorem, we have
\[(S(t')\psi_N)^p = N^{-p(s+\frac12)}\sum_{k=0}^p\binom{p}{k}(S(t')v_N)^k(S(t')w_N)^{p-k}.\]

As in the proof of Proposition \ref{prop:multilinear_sharp}, the Fourier support of the $k$-th term is contained in
\begin{equation}\label{eq:interval}
k[N,2N]+(p-k)[-2(p-1)N,-(p-1)N].
\end{equation}
If $k\le p-2$, then the upper bound of the interval \eqref{eq:interval} is
\[2kN-(p-k)(p-1)N \le 2(p-2)N-2(p-1)N = -2N,\]
while for $k=p$, the lower bound of the interval \eqref{eq:interval} is $pN$. Hence, for all $k\neq p-1$, the support is disjoint from the interval $[-1,1]$ once $N$ is sufficiently large. Therefore,
\begin{equation}\label{eq:u_N^p low}
P_{\le 1}u_N^{(p)}(t)=-N^{-p(s+\frac12)}\binom{p}{p-1}\int_0^t S(t-t')\phi(\partial_x)P_{\le 1}\left((S(t')v_N)^{p-1}S(t')w_N\right)\,dt'.
\end{equation}
The Fourier transform of \eqref{eq:u_N^p low} is given by
\begin{equation}\label{eq:FT u_N^p}
N^{-p(s+\frac12)}\binom{p}{p-1}\int_0^t e^{-i(t-t')\phi(\xi)}\mathcal F\left(\phi(\partial_x)P_{\le 1}\left((S(t')v_N)^{p-1}S(t')w_N\right)\right)(\xi)\,dt'.
\end{equation}
For the relevant interaction, the oscillatory factor takes the form
\[e^{-i(t-t')\phi(\xi)}e^{-it'((p-1)\phi(\xi_1)+\phi(\xi_p))}=e^{-it\phi(\xi)}e^{-it'\Theta},\]
where
\[\Theta=\phi(\xi)-\sum_{j=1}^p\phi(\xi_j).\]
On the support, we have $|\xi|\le 1$ and $|\xi_j|\sim N$, so
\[|\Theta|\le 1\]
for sufficiently large $N$. Fix
\[0 < t \le \min\left(T, \frac{\pi}{3}\right).\]
Since
\[\left|\int_0^t e^{-it'\Theta}\,dt'\right| \ge \Re\int_0^t e^{-it'\Theta}\,dt' = \int_0^t\cos(t'\Theta)\,dt' \ge \frac t2,\]
we have, for $|\xi| \le 1$,
\[|\eqref{eq:FT u_N^p}| \gtrsim t\,N^{-p(s+\frac12)} \left|\mathcal F\left(P_{\le 1}\phi(\partial_x)(u_1^{p-1}u_p)\right)(\xi)\right|.\]
Since $u_1=\cdots=u_{p-1}$, it follows from the proof of Proposition \ref{prop:multilinear_sharp} that
\[\|P_{\le 1}\phi(\partial_x)(u_1^{p-1}u_p)\|_{H^s(\mathbb R)}\gtrsim N^{p-1}.\]
Thus,
\[\|u_N^{(p)}(t)\|_{H^s(\mathbb R)} \ge \|P_{\le 1}u_N^{(p)}(t)\|_{H^s(\mathbb R)} \gtrsim t\,N^{-p(s+\frac12)}N^{p-1} = tN^\gamma,\]
where
\[\gamma:=\frac{p-2}{2}-ps.\]
Since $s<\frac{p-2}{2p}$, we have $\gamma>0$, and hence
\[\|u_N^{(p)}(t)\|_{H^s(\mathbb R)} \to \infty \quad \text{as } N \to \infty,\]
which contradicts, by \eqref{eq:frechet_bound_common},
\[\|u_N^{(p)}\|_{C([0,T];H^s(\mathbb R))} \lesssim \|\psi_N\|_{H^s(\mathbb R)}^p \sim 1.\]
Therefore, the flow map cannot be of class $C^p$ at the origin on $\mathbb R$.

\subsection{Lack of $C^p$-differentiability on $\mathbb T$}

We follow the same argument as before. Assume that $\Phi$ is a map from $H^s(\mathbb T)$ to $C([0,T];H^s(\mathbb T))$ of class $C^p$ near the origin. For an integer $N \gg 1$, let
\[\widehat{u}_j(n)=\mathbf1_{\{N,\dots,2N\}}(n),\quad 1 \le j \le p-1,\quad \widehat{u}_p(n)=\mathbf1_{\{1-2(p-1)N,\dots,1-(p-1)N\}}(n).\]
Define $v_N$, $w_N$, and $\psi_N$ as in \eqref{eq:v_Nw_N} so that
\[\|\psi_N\|_{H^s(\mathbb T)} \sim 1\]
uniformly in $N$. Define $u_N^{(p)}(t)$ as in the proof on $\mathbb R$. Considering the Fourier coefficient of $u_N^{(p)}(t)$ at the mode $n=1$, the Fourier support of the $k$-th term is contained in
\begin{equation}\label{eq:set}
k\{N,\dots,2N\}+(p-k)\{1-2(p-1)N,\dots,1-(p-1)N\}.
\end{equation}
If $k\le p-2$, then the upper bound of the set \eqref{eq:set} is
\[k(2N)+(p-k)(1-(p-1)N)\le 2(p-2)N+2(1-(p-1)N)=2-2N,\]
while for $k=p$, the lower bound of the set \eqref{eq:set} is $pN$. Hence, for all $k\neq p-1$, the support does not contain the mode $n=1$ once $N$ is sufficiently large. Therefore,
\[\widehat{u}_N^{(p)}(t)(1)=-N^{-p(s+\frac12)}\binom{p}{p-1}\int_0^t e^{-i(t-t')\phi(1)}\mathcal F\left(\phi(\partial_x)\left((S(t')v_N)^{p-1}S(t')w_N\right)\right)(1)\,dt'.\]
For the relevant interaction, the oscillatory factor takes the form
\[e^{-i(t-t')\phi(1)}e^{-it'(\phi(n_1)+\cdots+\phi(n_p))}=e^{-it\phi(1)}e^{-it'\Theta},\]
where
\[\Theta=\phi(n_1)+\cdots+\phi(n_p)-\phi(1),\qquad n_1+\cdots+n_p=1.\]
On the support, we have $\phi(1)=\frac12$ and $|n_j|\sim N$, so $|\phi(n_j)|\lesssim N^{-1}$ for each $j$ and
\[|\Theta|\le \frac34\]
for sufficiently large $N$. Fix
\[0 < t \le \min\left(T,\frac{\pi}{3}\right).\]
Since
\[\left|\int_0^t e^{-it'\Theta}\,dt'\right|\ge \Re\int_0^t e^{-it'\Theta}\,dt'=\int_0^t\cos(t'\Theta)\,dt'\ge \frac t2,\]
we have
\[|\widehat{u}_N^{(p)}(t)(1)| \gtrsim t\,N^{-p(s+\frac12)}\left|\mathcal F\left(\phi(\partial_x)(u_1^{p-1}u_p)\right)(1)\right|.\]
Since $u_1=\cdots=u_{p-1}$, it follows from the proof of Proposition \ref{prop:multilinear_sharp} that
\[\|\phi(\partial_x)(u_1^{p-1}u_p)\|_{H^s(\mathbb T)} \ge \langle1\rangle^s\left|\mathcal F\left(\phi(\partial_x)(u_1^{p-1}u_p)\right)(1)\right|\gtrsim N^{p-1}.\]
Thus,
\[\|u_N^{(p)}(t)\|_{H^s(\mathbb T)} \ge \langle1\rangle^s|\widehat{u}_N^{(p)}(t)(1)|\gtrsim t\,N^{-p(s+\frac12)}N^{p-1}=tN^\gamma,\]
where
\[\gamma:=\frac{p-2}{2}-ps.\]
Since $s<\frac{p-2}{2p}$, we have $\gamma>0$, and hence
\[\|u_N^{(p)}(t)\|_{H^s(\mathbb T)} \to \infty \quad \text{as } N \to \infty,\]
which contradicts, by \eqref{eq:frechet_bound_common},
\[\|u_N^{(p)}\|_{C([0,T];H^s(\mathbb T))} \lesssim \|\psi_N\|_{H^s(\mathbb T)}^p \sim 1.\]
Therefore, the flow map cannot be of class $C^p$ at the origin on $\mathbb T$.

\section{Proof of Theorem \ref{thm:GWP}}\label{sec:GWP}

For $s\ge 1$, the global well-posedness of \eqref{eq:gBBM} is already known from \cite{BBM}. Thus it remains to consider the cases $\frac14\le s<1$ when $p=3$, and $\frac12<s<1$ when $p=5$. As mentioned in the introduction, we mainly follow the Bona--Tzvetkov approach \cite{BT}.

Fix $T>0$ and $u_0\in H^s(\mathcal M)$. We decompose
\[u_0=w_0+v_0,\]
where $w_0$ is the low-frequency part of $u_0$ and $v_0$ is the high-frequency remainder. More precisely, if $\mathcal M=\mathbb R$, we define
\[\widehat{w_0}(\xi):=\mathbf 1_{\{|\xi|<N\}}(\xi)\widehat{u_0}(\xi),\qquad \widehat{v_0}(\xi):=\mathbf 1_{\{|\xi|\ge N\}}(\xi)\widehat{u_0}(\xi),\]
while if $\mathcal M=\mathbb T$, we define
\[\widehat{w_0}(n):=\mathbf 1_{\{|n|<N\}}(n)\widehat{u_0}(n),\qquad \widehat{v_0}(n):=\mathbf 1_{\{|n|\ge N\}}(n)\widehat{u_0}(n),\qquad n\in\mathbb Z.\]

We choose $N\gg1$ so that $\|v_0\|_{H^s(\mathcal M)}$ is sufficiently small depending on $T$. More precisely, we choose $N$ so that, for $p=3$,
\[\|v_0\|_{H^s(\mathcal M)}^2\lesssim T^{-1},\]
whereas, for $p=5$,
\[\|v_0\|_{H^s(\mathcal M)}^4\lesssim T^{-1}.\]
Since the proof of Theorem \ref{thm:unconditional_well_posedness} ensures that, for initial data $f\in H^s(\mathcal M)$, the local existence time satisfies
\begin{equation}\label{eq:local_T}
T_{\mathrm{loc}}\gtrsim \|f\|_{H^s(\mathcal M)}^{-(p-1)},
\end{equation}
this choice of $N$ guarantees that the corresponding solution with initial data $v_0$ exists on the whole interval $[0,T]$, that is, $v\in C([0,T];H^s(\mathcal M))$.

Thus, it suffices to consider the following Cauchy problem:
\begin{equation}\label{eq:perturbed_general}
\begin{cases}
(1-\partial_x^2)w_t+\left(w+(w+v)^p-v^p\right)_x=0,\\
w(0,x)=w_0(x).
\end{cases}
\end{equation}
Indeed, if $w\in C([0,T];H^s(\mathcal M))$ solves \eqref{eq:perturbed_general}, then $u:=v+w$ solves \eqref{eq:gBBM} on $[0,T]$ with initial data $u_0$.

We now consider the two cases separately.

\subsection*{Case 1: $p=3$}
We first derive an a priori $H^1$ bound for $w$. Assume $\frac14\le s<1$. Since $u=w+v$ solves \eqref{eq:gBBM}, the conserved quantity
\[M[u](t)=\frac12\int_{\mathcal M}u^2(t,x)\,dx+\frac14\int_{\mathcal M}u^4(t,x)\,dx\]
is constant in time. In particular,
\[\frac14\|u(t)\|_{L^4(\mathcal M)}^4\le M[u_0].\]
Since $s\ge \frac14$ and $H^{\frac14}(\mathcal M)\hookrightarrow L^4(\mathcal M)$, we obtain
\begin{equation}\label{eq:L4_bound_p3}
\|u(t)\|_{L^4(\mathcal M)}^4\lesssim \|u_0\|_{H^s(\mathcal M)}^2+\|u_0\|_{H^s(\mathcal M)}^4
\end{equation}
for all $t\in[0,T]$ as long as the solution $u$ exists.

From
\[(w+v)^3-v^3=w^3+3w(w+v)v,\]
a direct computation yields
\[\frac12\frac{d}{dt}\|w(t)\|_{H^1(\mathcal M)}^2=3\int_{\mathcal M} w_x\,w\,(w+v)\,v\,dx \lesssim \|w_x\|_{L^2}\|w\|_{L^\infty}\|w+v\|_{L^4}\|v\|_{L^4}.\]
Using the Sobolev embedding $H^1(\mathcal M)\hookrightarrow L^\infty(\mathcal M)$, the bound
\[\|v\|_{C([0,T];H^s(\mathcal M))}\lesssim \|v_0\|_{H^s(\mathcal M)} \lesssim \|u_0\|_{H^s(\mathcal M)},\]
and \eqref{eq:L4_bound_p3}, we obtain
\[\frac12\frac{d}{dt}\|w(t)\|_{H^1(\mathcal M)}^2\lesssim L\,\|w(t)\|_{H^1(\mathcal M)}^2,\]
where
\[L=L(\|u_0\|_{H^s(\mathcal M)}):=\|u_0\|_{H^s(\mathcal M)}\left(\|u_0\|_{H^s(\mathcal M)}^2+\|u_0\|_{H^s(\mathcal M)}^4\right)^{1/4}.\]
Note that $L$ is independent of time and depends only on the size of the initial data $u_0$. By Gr\"onwall's inequality,
\begin{equation}\label{eq:Gronwall_p3}
\|w(t)\|_{H^1(\mathcal M)}\lesssim \|w_0\|_{H^1(\mathcal M)}e^{Lt}
\end{equation}
as long as the solution exists.

To complete the proof, we apply the local well-posedness theory for \eqref{eq:perturbed_general} with $p=3$ (a slight modification of Theorem \ref{thm:unconditional_well_posedness}), obtaining a solution $w\in C([0,T_1];H^1(\mathcal M))$ for some $T_1>0$. On this interval, \eqref{eq:Gronwall_p3} shows that
\begin{equation}\label{eq:uniform bound}
\|w(T_1)\|_{H^1(\mathcal M)}\le C\|w_0\|_{H^1(\mathcal M)}e^{LT_1}\le C\|w_0\|_{H^1(\mathcal M)}e^{LT}.
\end{equation}
Note that on $[0,T_1]$, we have $v, w$ so that $u$ exists on $[0,T_1]$, thus \eqref{eq:Gronwall_p3} is available. Note also that the local existence time for the perturbed problem depends only on $\|w_0\|_{H^1(\mathcal M)}$, $\|u_0\|_{H^s(\mathcal M)}$, and $T$. 

We may restart the equation at time $T_1$ with initial data $w(T_1)$ and obtain a solution on $[T_1,T_1+\widetilde T]$, where $\widetilde T>0$ is independent of the step due to \eqref{eq:uniform bound}. Repeating this argument finitely many times, we cover the whole interval $[0,T]$. Consequently,
\[w\in C([0,T];H^1(\mathcal M))\subset C([0,T];H^s(\mathcal M)),\]
and hence $u=v+w\in C([0,T];H^s(\mathcal M))$. This proves global well-posedness for $p=3$ and $\frac14\le s<1$.

\subsection*{Case 2: $p=5$}

For this case, it suffices to derive an a priori $H^1$ bound for $w$. The rest follows analogously as in the case $p =3$. Assume $\frac12<s<1$. Since $u=w+v$ solves \eqref{eq:gBBM}, the conserved quantity
\[M[u](t)=\frac12\int_{\mathcal M}u^2(t,x)\,dx+\frac16\int_{\mathcal M}u^6(t,x)\,dx\]
is constant in time. In particular,
\[\frac16\|u(t)\|_{L^6(\mathcal M)}^6\le M[u_0].\]
Since $s>\frac12$ and $H^s(\mathcal M)\hookrightarrow L^6(\mathcal M)$, we obtain
\begin{equation}\label{eq:L6_bound_p5}
\|u(t)\|_{L^6(\mathcal M)}^6\lesssim \|u_0\|_{H^s(\mathcal M)}^2+\|u_0\|_{H^s(\mathcal M)}^6
\end{equation}
for all $t\in[0,T]$ as long as the solution $u$ exists.

From
\[(w+v)^5-v^5=w^5+5wv(w+v)\left((w+v)^2-(w+v)v+v^2\right),\]
a direct computation gives
\[\begin{aligned}
\frac12\frac{d}{dt}\|w(t)\|_{H^1(\mathcal M)}^2 =&~{} 5\int_{\mathcal M} w_x\,w\,\left((w+v)^3v-(w+v)^2v^2+(w+v)v^3\right)\,dx\\
\lesssim&~{} \|w_x\|_{L^2}\|w\|_{L^\infty}\left(\|w+v\|_{L^6}^3\|v\|_{L^\infty}+\|w+v\|_{L^6}^2\|v\|_{L^{12}}^2+\|w+v\|_{L^6}\|v\|_{L^9}^3\right).
\end{aligned}\]
Since $s>\frac12$, Sobolev embedding gives
\[\|v\|_{L^\infty(\mathcal M)}+\|v\|_{L^{12}(\mathcal M)}+\|v\|_{L^9(\mathcal M)}\lesssim \|u_0\|_{H^s(\mathcal M)}.\]
Together with \eqref{eq:L6_bound_p5}, this yields
\[\frac12\frac{d}{dt}\|w(t)\|_{H^1(\mathcal M)}^2\lesssim \widetilde L\,\|w(t)\|_{H^1(\mathcal M)}^2,\]
where
\[\begin{aligned}
\widetilde L=\widetilde L(\|u_0\|_{H^s(\mathcal M)}):=&~{}\|u_0\|_{H^s(\mathcal M)}\left(\|u_0\|_{H^s(\mathcal M)}^2+\|u_0\|_{H^s(\mathcal M)}^6\right)^{1/2}\\
+&~{}\|u_0\|_{H^s(\mathcal M)}^2\left(\|u_0\|_{H^s(\mathcal M)}^2+\|u_0\|_{H^s(\mathcal M)}^6\right)^{1/3}\\
+&~{}\|u_0\|_{H^s(\mathcal M)}^3\left(\|u_0\|_{H^s(\mathcal M)}^2+\|u_0\|_{H^s(\mathcal M)}^6\right)^{1/6}.
\end{aligned}\]
Again, $\widetilde L$ is independent of time and depends only on the size of the initial data. By Gr\"onwall's inequality,
\[\|w(t)\|_{H^1(\mathcal M)}\lesssim \|w_0\|_{H^1(\mathcal M)}e^{\widetilde Lt}\]
as long as the solution exists. Following the analogous argument in the case $p=3$, we complete the proof.

\providecommand{\bysame}{\leavevmode\hbox to3em{\hrulefill}\thinspace}
\providecommand{\MR}{\relax\ifhmode\unskip\space\fi MR }
\providecommand{\MRhref}[2]{%
  \href{http://www.ams.org/mathscinet-getitem?mr=#1}{#2}
}
\providecommand{\href}[2]{#2}

\end{document}